\documentclass{article}
\usepackage{spconf,amsmath,graphicx}
\usepackage{amsmath,amsthm,amssymb,latexsym,color, dsfont}
\usepackage{wrapfig}
\usepackage[ruled,vlined]{algorithm2e}

\usepackage{mathabx}

\usepackage{bbm}
\usepackage{mathtools}
\usepackage{comment}

\usepackage{nicefrac}

\def\E{\mathbb{E}}
\def\R{\mathbb{R}}

\renewcommand{\epsilon}{\varepsilon}

\newtheorem{theorem}{Theorem}

\newtheorem{lemma}{Lemma}

\newtheorem{remark}{Remark}

\newcommand{\Lb}{\mathbf{L}}

\title{Scattering Statistics of Generalized Spatial Poisson Point Processes}
\name{Michael Perlmutter$^{\dagger}$ \qquad Jieqian He$^{\ddagger,\star}$ \qquad Matthew Hirn$^{\ddagger,\star\star}$ \thanks{M.H. acknowledges support from NSF DMS \#1845856 (partially supporting J.H.), NIH NIGMS \#R01GM135929, and DOE \#DE-SC0021152.}}
  
\address{$^{\dagger}$ University of California, Los Angeles, Department of Mathematics \\
$^\ddagger$ Michigan State University, Department of Computational Mathematics, Science \& Engineering \\
$\star$ Michigan State University, Department of Statistics and Probability \\
$\star\star$ Michigan State University, Department of Mathematics}

\begin{document}
%
\maketitle
\begin{abstract}

We present a machine learning model for the analysis of randomly
generated discrete signals, modeled as the points of an inhomogeneous, compound Poisson point process. Like the
wavelet scattering transform introduced by  Mallat, our construction is
naturally invariant to translations and reflections, but it decouples the
roles of scale and frequency, 
replacing wavelets with Gabor-type measurements. We show that, with suitable 
nonlinearities, our measurements distinguish Poisson point processes from
common self-similar processes, and separate different types of Poisson point processes. 
\end{abstract}
\section{Introduction}
\label{sec:intro}

Convolutional neural networks (CNNs) have obtained impressive results for a number of learning tasks in which the underlying signal data can be modelled as a stochastic process, including texture discrimination \cite{mallat:rotoScat2013}, texture synthesis \cite{NIPS2015_5633, arXiv:1806.08002}, time-series analysis \cite{pmlr-v80-binkowski18a}, and wireless networks \cite{arXiv:1812.08265}. 
In many scenarios, it is natural to model the signal data as the points of a (potentially complex) spatial point process. Furthermore, there are numerous other fields, including stochastic geometry \cite{haenggi:stochGeoWireless2009}, forestry \cite{genet:Forestry2014}, geoscience \cite{schoenberg2016} and genetics \cite{fromionGene2013}, in which spatial point processes are used to model the underlying generating process of certain phenomena (e.g., earthquakes). 
This motivates us to consider the capacity of CNNs to capture the statistical properties of such processes.  


The Wavelet scattering transform \cite{mallat:scattering2012} is a model for CNNs, which consists of an alternating cascade of linear wavelet transforms and complex modulus nonlinearities. It has provable stability and invariance properties and has been used to achieve near state of the art results in fields such as audio signal processing 
\cite{anden:deepScatSpectrum2014}, computer vision 
\cite{oyallon:scatObjectClass2014}, and quantum chemistry 
\cite{brumwell:steerableScatLiSi2018}. 
 In this paper, we examine a generalized scattering transform that utilizes a broader class of filters (which includes wavelets). We primarily focus on filters with small support, which is similar to those used in most CNNs.

Expected wavelet scattering moments for stochastic processes with stationary increments were introduced in \cite{bruna:scatMoments2015}, where it is shown that such moments capture important statistical information of one-dimensional Poisson processes, fractional Brownian motion,  $\alpha$-stable L\'{e}vy processes, and a number of other stochastic processes. In this paper, we extend the notion of scattering moments to our generalized  architecture, and generalize many of the results from \cite{bruna:scatMoments2015}. However, the main contributions contained here consist of new results for more general spatial point processes, including  inhomogeneous Poisson point processes, which  are not stationary and do not have stationary increments. The collection of expected scattering moments is a non-parametric model for these processes, which we show captures important summary statistics.

In Section \ref{sec: expected scattering moments} we will define our expected scattering moments. Then, in Sections \ref{sec: first-order scattering poisson}  and \ref{sec: scattering non-poisson}  we will analyze these moments for certain generalized  Poisson point processes and  self-similar processes.  We will present  numerical examples in Section \ref{sec: first-order numerics}, and provide a short  conclusion in section \ref{sec: conclusion}.

\section{Expected Scattering Moments 
}
\label{sec: expected scattering moments}
Let $\psi \in \Lb^2 (\R)$ be a compactly supported mother wavelet with dilations $\psi_j (t) = 2^{-j} \psi (2^{-j} t)$ for $j\in\mathbb{Z}$, and let $X(t), t \in \R,$ be a stochastic process with stationary increments. The first-order wavelet scattering moments are defined in \cite{bruna:scatMoments2015} as $S X (j) = \E [ | \psi_j \ast X| ]$, where the expectation does not depend on $t$ since $X(t)$ has stationary increments and $\psi_j$ is a wavelet which implies $X \ast \psi_j (t)$ is stationary.
Much of the  analysis of  in \cite{bruna:scatMoments2015} relies on the fact that these moments can be rewritten as $S X (j) = \E [|\overline{\psi}_j \ast dX|]$, where 
$d\overline{\psi}_j = \psi_j$. This motivates us to define scattering moments  as the integration of a filter, 
against a random signed measure $Y(dt).$  

To that end, let $w \in \Lb^2 (\R^d)$ be a continuous window function with support contained in $[0,1]^d$. Denote by $w_s(t) = w\left(\frac{t}{s}\right)$ the dilation of $w$, 
and set $g_{\gamma} (t)$ to be the Gabor-type filter with scale $s > 0$ and central frequency $\xi \in \R^d$, 
\begin{equation} \label{eqn: gabor filter}
    g_{\gamma} (t) = w_s(t) e^{i \xi \cdot t}, \quad \gamma = (s, \xi), \enspace t \in \R^d \, .
\end{equation}
Note that with an appropriately chosen window function $w,$ \eqref{eqn: gabor filter} includes dyadic wavelet families in the case that  $s = 2^j$ and 
$| \xi | = C/s$ . However, it also includes many other filters, such as Gabor filters  used in the windowed Fourier transform.

Let $Y(dt)$ be a random signed measure and assume that $Y$ is $T$-periodic for some $T > 0$ in the sense that for any Borel set $B$ we have  $Y (B) = Y (B + T e_i)$ , for all $ 1 \leq i \leq d$  (where $\{ e_i \}_{i \leq d}$ is the standard orthonormal basis for $\R^d$). For $f \in \Lb^2 (\R^d)$, set $f \ast Y(t) \coloneqq \int_{\R^d} f(t-u) Y(du)$. We define the first-order and second-order expected scattering moments, $1\leq p,p' <\infty,$ at location $t$ as 
\begin{align} \label{eqn: first-order scattering moment}
    S_{\gamma, p} Y(t) &\coloneqq \E [ |g_{\gamma} \ast Y(t)|^p ] \quad\text{and}\\
    S_{\gamma, p, \gamma', p'} Y (t) &\coloneqq \E \left[ ||g_{\gamma} \ast Y|^p \ast g_{\gamma'} (t) |^{p'} \right] \, .
\end{align}
Note $Y(dt)$ is not assumed to be stationary, which is why these moments depend on $t$. 
Since $Y(dt)$ is periodic, we may also define time-invariant scattering coefficients by 
\begin{align*}
    SY(\gamma,p)&\coloneqq \frac{1}{T^d}\int_{[0,T]^d}  S_{\gamma, p} Y(t) dt, \quad \text{and}\\
    SY(\gamma,p,\gamma',p')&\coloneqq \frac{1}{T^d}\int_{[0,T]^d}  S_{\gamma, p,\gamma',p'} Y(t) dt
\end{align*}

In the following sections, we analyze these moments for arbitrary frequencies $\xi$ and small scales $s$, thus allowing the filters $g_{\gamma}$ to serve as a model for the learned filters in CNNs. In particular, we will analyze the asymptotic behavior of the scattering moments as $s$ decreases to zero.

\section{ Scattering Moments of Generalized Poisson Processes}
\label{sec: first-order scattering poisson}
In this section, we let $Y (dt)$ be an inhomogeneous, compound spatial Poisson point process. 
Such processes generalize ordinary Poisson point processes by incorporating variable charges (heights) at the points of the process 
and a non-uniform intensity for the locations of the points. 
 They thus provide a flexible family of point processes that can be used to model many different phenomena. 
In this section, we provide a review of such processes and analyze their first and second-order scattering moments.


Let $\lambda (t)$ be a continuous, periodic function on $\R^d$ with
\begin{equation} \label{eqn: intensity bound}
    0 < \lambda_{\min} := \inf_t \lambda (t) \leq \| \lambda \|_{\infty} < \infty \, ,
\end{equation}
and define its first and second order moments by 
\begin{equation*}
    m_p (\lambda) \coloneqq \frac{1}{T^d} \int_{[0,T]^d} \lambda(t)^2 \, dt, \quad p=1,2 \, .
\end{equation*}
A random measure $N(dt) \coloneqq \sum_{j=1}^\infty \delta_{t_j}(dt)$ is called an inhomogeneous Poisson point process with intensity function $\lambda (t)$ if for any Borel set $B \subset \R^d$, 
$$P(N(B) = n) = e^{- \Lambda(B)} \frac{\bigl (\Lambda(B) \bigr)^n}{n!}, \enspace \Lambda(B) = \int_B \lambda(t) \, dt,$$
and, in addition, 
$N(B)$ is independent of $N (B')$ for all  $B'$ that do not intersect $B$. Now let $(A_j)_{j=1}^{\infty}$ be a sequence of i.i.d. random variables independent of $N$.
 An inhomogeneous, compound Poisson point process $Y(dt)$ is given by 
\begin{equation} \label{eqn: poisson point process}
    Y (dt) = \sum_{j=1}^{\infty} A_j \delta_{t_j} (dt) \, .
\end{equation}
For a further overview of these processes, 
we refer the reader to Section 6.4 of \cite{Daley}. 

\subsection{First-order Scattering Asymptotics}
\label{sec: first-order scattering theory}

Computing the convolution of $g_{\gamma}$ with $Y(dt)$  gives
\begin{equation*}
    (g_{\gamma} \ast Y) (t) = \int_{\R^d} g_{\gamma}(t-u) \, Y (du) = \sum_{j=1}^{\infty} A_j g_{\gamma} (t - t_j) \, ,
\end{equation*}
which can be interpreted as a waveform $g_{\gamma}$ emitting from each location $t_j$. Invariant scattering moments aggregate the random interference patterns in $|g_{\gamma} \ast Y|$. The results below show that the expectation of these interference patterns 
encode important statistical information related to the point process. 

For notational convenience, we let
\begin{equation*}
    \Lambda_s(t) := \Lambda \left([t-s, t]^d\right) = \int_{[t-s, t]^d} \lambda (u) \, du
\end{equation*}
denote the expected number of points of $N$ in the support of $g_{\gamma} (t - \cdot)$. 
By conditioning on $N \left([t-s, t]^d\right)$, the number of points in the support of $g_{\gamma}$, and using the fact that
\begin{equation*} 
    \mathbb{P} \left[N \left([t-s, t]^d\right) > m\right] = \mathcal{O} \left(\left(s^d \| \lambda \|_{\infty}\right)^{m+1}\right)
\end{equation*}
one may obtain the following theorem.\footnote{A proof of Theorem 1, as well as the proofs of other theorems stated in this paper, can be found in the appendix}

\begin{theorem} \label{thm: scattering 1st order taylor}
Let $\E [|A_1|^p] < \infty$, and $\lambda (t)$ be a periodic continuous intensity function satisfying \eqref{eqn: intensity bound}. Then for every $t\in\mathbb{R}^d,$ every $\gamma=(s,\xi)$ such that $s^d\|\lambda\|_\infty<1,$ and every $m\geq 1,$ 
\begin{equation} \label{eqn: scattering taylor}
    S_{\gamma, p} Y (t) \approx \sum_{k=1}^m e^{-\Lambda_s (t)} \frac{(\Lambda_s(t))^k}{k!} \E \left[ \left| \sum_{j=1}^k A_j w (V_j) e^{i s \xi \cdot V_j} \right|^p \right], 
\end{equation}
where the error term $\epsilon (m, s, \xi, t)$ satisfies
\begin{equation} \label{eqn: error bound}
    | \epsilon (m, s, \xi, t) | \leq C_{m,p} \frac{\| \lambda \|_{\infty}}{\lambda_{\min}} \| w \|_p^p \E [|A_1|^p] \| \lambda \|_{\infty}^{m+1} s^{d(m+1)} \, 
\end{equation}
and $V_1, V_2, \ldots$ is an i.i.d. sequence of random variables, independent of the $A_j$, taking values in the unit cube $[0,1]^d$ and with density
    $p_V (v) = \frac{s^d}{\Lambda_s(t)} \lambda (t - vs)$ for $v\in[0,1]^d.$
\end{theorem}

If we set $m=1,$ and let $s\rightarrow 0,$ then one may use the fact that   a small cube $[t-s,t]^d$ has at most one point of $N$ with overwhelming probability to obtain the following result.
\begin{theorem} \label{thm: scattering 1st order small scale limit}
Let $Y(dt)$ satisfy the same assumptions as in Theorem \ref{thm: scattering 1st order taylor}. Let $\gamma_k = (s_k, \xi_k)$ be a sequence of scale and frequency pairs such that $\lim_{k \rightarrow \infty} s_k = 0$. Then
\begin{equation} \label{eqn: location scattering small scales}
    \lim_{k \rightarrow \infty} \frac{S_{\gamma_k, p} Y (t)}{s_k^d} = \lambda (t) \E [ |A_1|^p ] \| w \|_p^p \, ,
\end{equation}
for all $t$, and consequently
\begin{equation} \label{eqn: invariant small scales}
    \lim_{k \rightarrow \infty} \frac{S Y (\gamma_k, p)}{s_k^d} = m_1 (\lambda) \E [ |A_1|^p ] \| w \|_p^p.
\end{equation}
\end{theorem}


This theorem shows that for small scales the scattering moments $S_{\gamma, p} Y(t)$ encode the intensity function $\lambda (t)$, up to factors depending upon the summary statistics of the charges $(A_j)_{j=1}^{\infty}$ and the window $w$. 
Thus even a one-layer location-dependent scattering network yields considerable information regarding the underlying data generation process. 


In the case of ordinary (non-compound) homogeneous Poisson processes, Theorem \ref{thm: scattering 1st order small scale limit} recovers the constant intensity. 
For general $\lambda (t)$ and invariant scattering moments, the role of higher-order moments of $\lambda (t)$ is highlighted by considering higher-order expansions (e.g., $m > 1$) in \eqref{eqn: scattering taylor}. The next theorem considers second-order expansions and illustrates their dependence on the second moment of $\lambda (t)$. 

\begin{theorem} \label{thm: scattering 2nd order taylor}
Let $Y$  satisfy the same assumptions as in Theorem \ref{thm: scattering 1st order taylor}. If  $(\gamma_k)_{k \geq 1} = (s_k, \xi_k)_{k \geq 1}$, is a sequence 
such that $\lim_{k \rightarrow \infty} s_k = 0$ and $\lim_{k \rightarrow \infty} s_k \xi_k = L \in \R^d$, then
\begin{align} 
    \lim_{k \rightarrow \infty} &\left( \frac{S Y (\gamma_k, p)}{s_k^{2d} \E [|A_1|^p] \E [|V_{k}|^p]} - \frac{1}{T^d} \int_{[0,T]^d} \frac{\Lambda_{s_k} (t)}{s_k^{2d}} \, dt \right) \nonumber \\[5pt]
    ={} &m_2 (\lambda) \left( \frac{\E\left[ \left|A_1 w (U_1) e^{i L \cdot U_1} + A_2 w (U_2) e^{i L \cdot U_2} \right|^p \right]}{2 \| w \|_p^p \E [|A_1|^p]} \right) \, , \label{eqn: 2nd order taylor asymptotics}
\end{align}
where  $U_1$, $U_2$ are independent uniform random variables on $[0,1]^d$; and $(V_{k})_{k \geq 1}$ is a sequence of random variables independent of the $A_j$ taking values in the unit cube with respective densities, 
    $p_{V_{k}} (v) = \frac{s_k^d}{\Lambda_{s_k} (t)} \lambda (t - vs_k)$  for $v\in[0,1]^d$.
\end{theorem}

We note that the scale normalization on the left hand side of \eqref{eqn: 2nd order taylor asymptotics} is $s^{-2d}$, compared to a normalization of $s^{-d}$ in Theorem \ref{thm: scattering 1st order small scale limit}. Thus, intuitively, \eqref{eqn: 2nd order taylor asymptotics}  is capturing information at moderately small scales that are larger than the scales considered in Theorem \ref{thm: scattering 1st order small scale limit}. 
Unlike Theorem \ref{thm: scattering 1st order small scale limit}, which gives a way to compute $m_1 (\lambda)$, Theorem \ref{thm: scattering 2nd order taylor} does not allow one to compute $m_2 (\lambda)$ since it would require knowledge of $\Lambda_{s_k} (t)$ in addition to the distribution from which the charges $(A_j)_{j=1}^{\infty}$ are drawn. However, Theorem \ref{thm: scattering 2nd order taylor} does show that at moderately small scales the invariant scattering coefficients depend non-trivially on the second moment of $\lambda (t)$. Therefore, they can be used to distinguish between, for example, an inhomogeneous Poisson point process with intensity function $\lambda (t)$ and a homogeneous Poisson point process with constant intensity. 

\subsection{Second-Order Scattering Moments of Generalized Poisson Processes}
\label{sec: second-order scattering poisson}

Our next result shows that second-order scattering moments 
 encode higher-order moment information about the 
 $(A_j)_{j=1}^{\infty}.$

\begin{theorem} \label{thm: 2nd order asymptotics}
Let $Y(dt)$  satisfy the same assumptions as in Theorem \ref{thm: scattering 1st order taylor}.
Let $\gamma_k = (s_k, \xi_k)$ and $\gamma_k' = (s_k', \xi_k')$ be sequences of scale-frequency pairs with $s_k' = cs_k$ for some $c > 0$ and $\lim_{k \rightarrow \infty} s_k \xi_k = L \in \R^d$.
Let $1\leq p, p' <\infty$ and $q = pp'.$ Assume $\mathbb{E}|A_1|^q<\infty,$ and let $K \coloneqq \left\| g_{c, L/c} \ast |g_{1, 0}|^p \right\|_{p'}^{p'}$.  Then,
\begin{align} \label{eqn: 2nd order loc dep asymp}
    \lim_{k \rightarrow \infty} \frac{S_{\gamma_k, p, \gamma_k', p'} Y (t)}{s_k^{d(p' + 1)}} &= K \lambda (t) \E [ |A_1|^q ], \quad\text{and} \\
 \label{eqn: 2nd order loc dep asymp inv}
    \lim_{k \rightarrow \infty} \frac{SY (\gamma_k, p, \gamma_k', p')}{s_k^{d(p'+1)}} &= K m_1 (\lambda) \E [ |A_1|^q ] \, .
\end{align}
\end{theorem}



Theorem \ref{thm: scattering 1st order small scale limit} shows first-order scattering moments with $p = 1$ are not able to distinguish between different types of Poisson point processes at very small scales if the charges have the same first moment. However, Theorem \ref{thm: 2nd order asymptotics} shows second-order scattering moments encode higher-moment information about the charges, and thus are better able to distinguish them (when used in combination with the first-order coefficients). In Sec.~\ref{sec: scattering non-poisson},  we will see first-order invariant scattering moments can distinguish Poisson point processes from self-similar processes if $p = 1,$ but may fail to do so for larger values of $p.$


\section{
Comparison to Self-Similar Processes}
\label{sec: scattering non-poisson}

We will show first-order invariant scattering moments can distinguish between Poisson point processes and certain self-similar processes, such as  $\alpha$-stable processes, $1<\alpha\leq 2,$ or  fractional Brownian motion (fBM). 
These results generalize those in \cite{bruna:scatMoments2015} both by considering more general filters and general $p^{\text{th}}$ scattering moments.

For a stochastic process $X(t),$ $t \in \R$, we consider the convolution of the filter $g_{\gamma}$ with the noise $dX$ defined by
 $   g_{\gamma} \ast dX (t) \coloneqq \int_{\R} g_{\gamma} (t-u) \, dX(u),$ 
and define (in a slight abuse of notation) the first-order scattering moments at time $t$  by 
    $S_{\gamma, p} X (t) \coloneqq \E [ | g_{\gamma} \ast dX (t)|^p ] \, .$
In the case where $X(t)$ is a compound, inhomogeneous Poisson (counting) process, $Y = dX$ will be a compound Poisson random measure and these scattering moments   will coincide with those defined in \eqref{eqn: first-order scattering moment}. 

The following theorem analyzes the small-scale first-order scattering moments when $X$ is either an $\alpha$-stable process, 
or an fBM. 
It shows the small-scale asymptotics of the corresponding scattering moments are guaranteed to differ from those of a Poisson point process when $p=1.$ We also note that both $\alpha$-stable processes and fBM have stationary increments and thus 
    $S_{\gamma, p} X (t) = SX (\gamma, p)$ for all $t$. 

\begin{theorem} \label{thm: alpha stable}
Let $1\leq p< \infty,$ and  let $\gamma_k = (s_k, \xi_k)$ be a sequence of scale-frequency pairs with $\lim_{k \rightarrow \infty} s_k = 0$ and $\lim_{k \rightarrow \infty} s_k \xi_k = L \in \R$. Then, if $X(t)$ is a symmetric $\alpha$-stable process,  $p < \alpha \leq 2,$ we have
\begin{equation*}
    \lim_{k \rightarrow \infty} \frac{S X (\gamma_k, p)}{s_k^{p/\alpha}} = \E \left[ \left| \int_0^1 w(u) e^{i L u} \, d X(u)  \right|^p \right] \, .
\end{equation*}
Similarly, if $X(t)$ is an fBM with Hurst parameter $H\in(0,1)$ and $w$ has bounded variation on $[0,1],$ then
\begin{equation*}
    \lim_{k \rightarrow \infty} \frac{SX (\gamma_k, p)}{s_k^{pH}} = \E \left[ \left| \int_0^1 w(u) e^{i L u} \, d X(u)  \right|^p \right] \, .
\end{equation*}
\end{theorem}

This theorem shows that first-order invariant scattering moments distinguish inhomogeneous, compound Poisson processes from both $\alpha$-stable processes and fractional Brownian motion except in the cases where $p = \alpha$ or $p = 1/H$. In particular, these measurements  distinguish Brownian motion,  from a Poisson point process except in the case where $p = 2$. 

\begin{figure}
    \centering
    \includegraphics[width = 0.15\textwidth, trim = {0 40pt 0 0 }, clip]{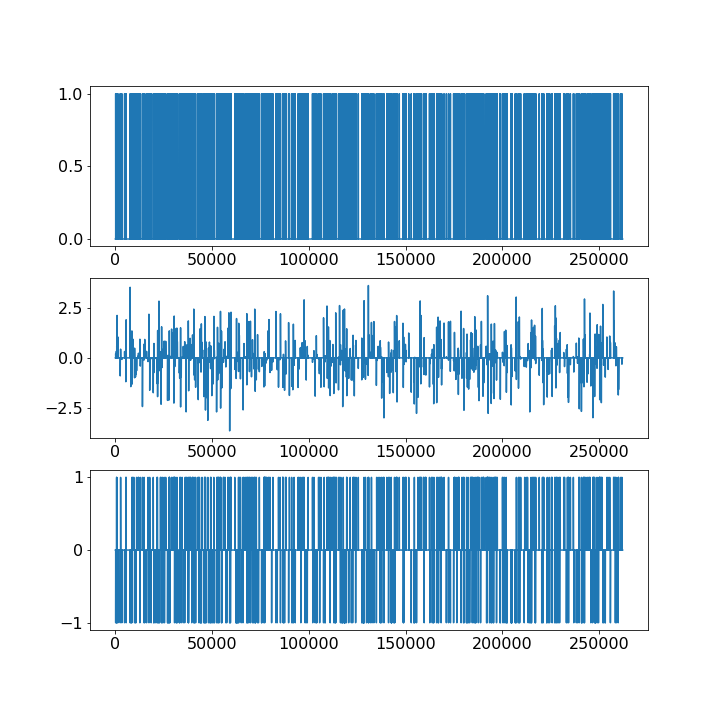}
    \includegraphics[width = 0.15\textwidth]{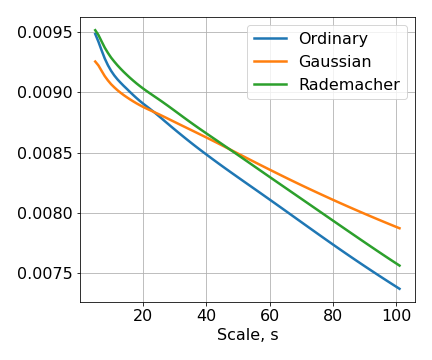}
    \includegraphics[width = 0.15\textwidth]{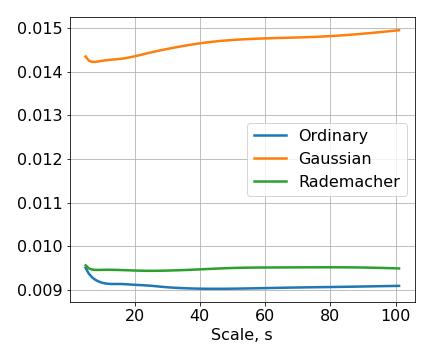}
    \caption{First-order invariant scattering moments of homogeneous compound Poisson point processes with the same intensity $\lambda_0$ and different $A_i$. \textbf{Left:} Realizations of the process with arrival rates given by \textit{Top:} $A_i=1$ \textit{Middle:} $A_i$ are normal random variables \textit{Bottom:} $A_i$ are Rademacher random variables. \textbf{Middle:}
    Plots of normalized first-order scattering $\frac{SY(s,\xi,1)}{s\|w\|_1}$ moments with $p=1$.\textbf{Right:} Plots of normalized first-order scattering $\frac{SY(s,\xi,2)}{s\|w\|_2^2}$ moments with $p=2$.}
    \label{fig:1}
\end{figure}

\section{Numerical Illustrations}
\label{sec: first-order numerics}

\begin{figure}
    \centering
    \includegraphics[width = 0.15\textwidth, trim = {0 -200pt 0 0}, clip]{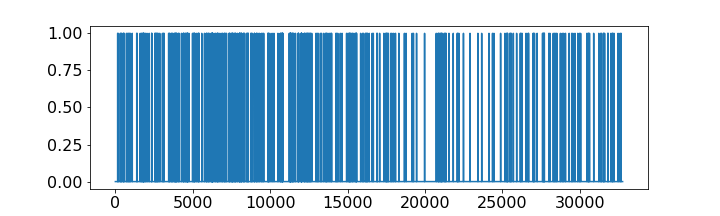}
    \includegraphics[width = 0.15\textwidth]{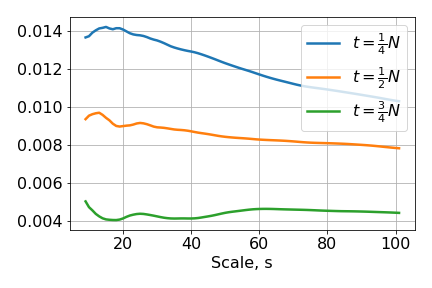}
    \caption{First-order scattering moments for inhomogeneous Poisson point processes. \textbf{Left:}  Sample realization with $\lambda(t) = 0.01 (1 + 0.5 \sin(\frac{2 \pi t}{N}))$. \textbf{Right:}  Time-dependent scattering moments $\frac{S_{\gamma, p}Y( t)}{s\Vert w\Vert_{p}^{p}}$ at $t_1 = \frac{N}{4}$, $t_2 = \frac{N}{2}$, $t_3 = \frac{3N}{4}$. Note that the scattering coefficients at times $t_1,t_2,t_3$ converges to $\lambda(t_1) = 0.015$, $\lambda(t_2) = 0.01$, $\lambda(t_3) = 0.005$. }
    \label{fig:4}
\end{figure}

\begin{figure}
    \centering
    \includegraphics[width = 0.15\textwidth, trim = {0 -100pt 0 0}, clip]{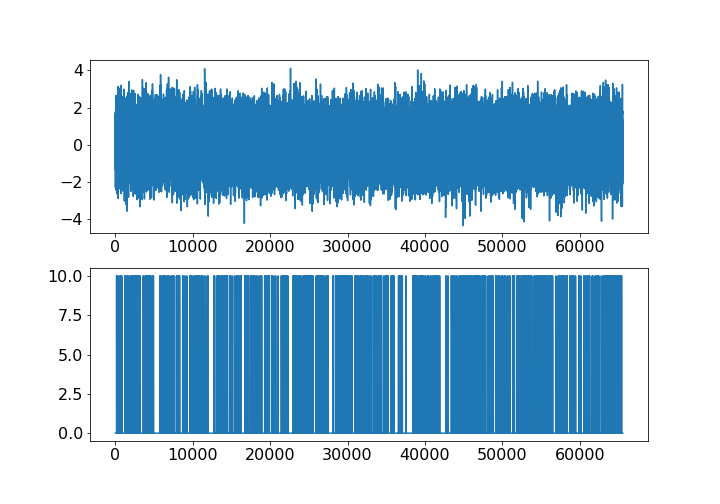}
    \includegraphics[width = 0.15\textwidth]{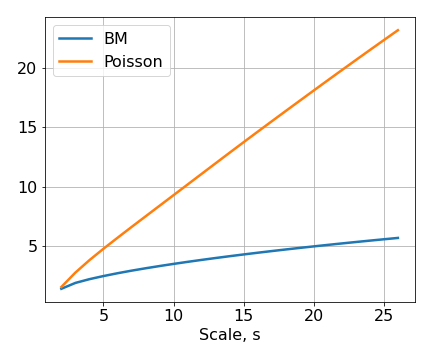}
    \includegraphics[width = 0.15\textwidth]{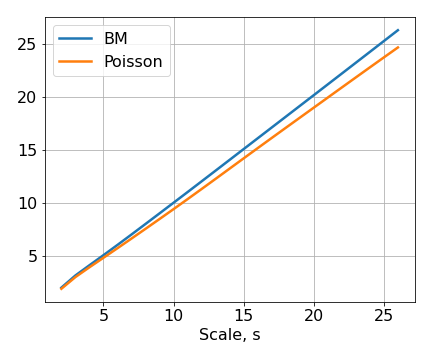}
    \caption{First-order invariant scattering moments for standard Brownian motion and Poisson point process. \textbf{Left:} Sample realizations  \textit{Top:} Brownian motion. \textit{Bottom:} Ordinary Poisson point process. \textbf{Middle:} Normalized scattering moments $\frac{SY_{\text{poisson}}(x,\xi,p)}{\lambda\mathbb{E}|A_1|^p\|w\|_p^p}$ and $\frac{SX_{\text{BM}}(s,\xi,p)}{\lambda\mathbb{E}|Z|^p\|w\|_p^p}$  for Poisson and BM with $p = 1$. \textbf{Right:} The same but with $p=2$.}
    \label{fig:5}
\end{figure}
We carry out several experiments to numerically validate the previously stated results. 
In all of our experiments, we hold the frequency $\xi$ constant while letting $s$ decrease to zero. 

\noindent\textbf{Compound Poisson point processes with the same intensities:}
We generated three  homogeneous compound Poisson point processes, all with intensity $\lambda (t) \equiv \lambda_0 = 0.01$, where the charges $A_{1,j}$, $A_{2,j}$, and $A_{3,j}$  are chosen so that $A_{1,j}=1$ uniformly, $A_{2,j}\sim \mathcal{N}(0, \sqrt{\frac{\pi}{2}}) $, and $A_{3,j}$ are Rademacher random variables. 
The charges of the three signals have the same first moment $\mathbb{E}[|A_{i,j}|] = 1$ and different second moment with $\mathbb{E}[|A_{1,j}|^2] = \mathbb{E}[|A_{3,j}|^2] = 1$ and  $\mathbb{E}[|A_{2,j}|^2] = \frac{\pi}{2}$. As predicted by  Theorem \ref{thm: scattering 1st order small scale limit}, Figure \ref{fig:1} shows first-order scattering moments will not be able to distinguish between the three processes with $p=1$, but  will distinguish the process with Gaussian charges from the other two when $p=2$. 

\noindent\textbf{Inhomogeneous, non-compound Poisson point processes:}
We also consider an inhomogeneous, non-compound Poisson point processes with intensity function $\lambda (t) = 0.01 (1 + 0.5 \sin(\frac{2 \pi t}{N}))$ (where we  estimate $S_{\gamma, p}Y( t)$, by averaging over 1000 realizations). Figure \ref{fig:4} plots the scattering moments for the inhomogeneous process at different times, and shows they align with the true intensity function.

\noindent\textbf{Poisson point process and self similar process:}
We consider a Brownian motion compared to a Poisson point process with intensity $\lambda = 0.01$ and charges $(A)_{j=1}^{\infty} \equiv 10$. Figure \ref{fig:5} shows the convergence rate of the first-order scattering moments can distinguish these processes when $p=1$ but not when $p=2.$
\section{Conclusion}
\label{sec: conclusion}
We have constructed Gabor-filter scattering transforms for random measures on $\mathbb{R}^d.$ 
Our work is closely related to \cite{bruna:scatMoments2015} but considers more general classes of filters and point processes (although  we note that \cite{bruna:scatMoments2015} provides a more detailed analysis of self-similar processes).
In future work, it  would be interesting to explore the use of these measurements for tasks such as, e.g., synthesizing new signals. 

\clearpage
\bibliographystyle{plain}

\newpage

\appendix

\section{Proof of Theorem 1}

To prove Theorem \ref{thm: scattering 1st order taylor} we will need the following lemma.

\begin{lemma} \label{lem: poisson variable exp tail}
Let $Z$ be a Poisson random variable with parameter $\lambda$. Then for all $\alpha \in \R$, $m \in \mathbb{N}$, $0<\lambda<1$, we have 
\begin{equation*}
    \E \left[ Z^{\alpha} \mathbbm{1}_{\{Z > m\}} \right] = \sum_{k=m+1}^{\infty} e^{-\lambda} \frac{\lambda^k}{k!} k^{\alpha} \leq C_{m, \alpha} \lambda^{m+1}.
\end{equation*}
\end{lemma}

\begin{proof}
For $0 < \lambda < 1$ and $k \in \mathbb{N}$, $e^{-\lambda} \lambda^k \leq 1$. Therefore,
\begin{align*}
    \E \left[ Z^{\alpha} \mathbbm{1}_{\{Z > m\}} \right] &= \sum_{k=m+1}^{\infty} e^{-\lambda} \frac{\lambda^k}{k!} k^{\alpha} \\
    &= \lambda^{m+1} \sum_{k=0}^{\infty} e^{-\lambda} \frac{\lambda^k}{(k + m + 1)!} (k + m + 1)^{\alpha} \\
    &\leq \lambda^{m+1} \sum_{k=0}^{\infty} \frac{(k + m + 1)^{\alpha}}{(k + m + 1)!} \\
    &= C_{\alpha, m} \lambda^{m+1} \, .
\end{align*}
\end{proof}

\begin{proof}[The proof of Theorem \ref{thm: scattering 1st order taylor}]
Recalling the definitions of $Y (dt)$ and $S_{\gamma, p} Y (t)$, and setting $N_s(t) = N\left([t-s, t]^d\right)$, we see
\begin{align*}
    S_{\gamma, p} Y (t) &= \E \left[ \left| \int_{[s-t, t]^d} w \left( \frac{t-u}{s} \right) e^{i \xi \cdot (t-u)} \, Y (du) \right|^p \right] \\
    &= \E \left[ \left| \sum_{j=1}^{N_s(t)} A_j w \left( \frac{t-t_j}{s} \right) e^{i \xi \cdot (t-t_j)} \right|^p \right] \, ,
\end{align*}
where $t_1, t_2, \ldots t_{N_s(t)}$ are the points $N(t)$ in $[t-s, t]^d$. Conditioned on the event that $N_s(t) = k$, the locations of the $k$ points on $[t-s, t]^d$ are distributed as i.i.d. random variables $Z_1, \ldots, Z_k$ taking values in $[t-s, t]^d$ with density
\begin{equation*}
    p_Z (z) = \frac{\lambda (z)}{\Lambda_s (t)} \, , \quad z \in [t-s, t]^d \, .
\end{equation*}
Therefore, the random variables 
\begin{equation*}
    V_i := \frac{t - Z_i}{s}
\end{equation*}
take values in the unit cube $[0,1]^d$ and have density
\begin{equation*}
    p_V (v) = \frac{s^d}{\Lambda_s(t)} \lambda (t - vs) \, , \quad v \in [0,1]^d \, .
\end{equation*}
Note that in the special case that $N$ is homogeneous, i.e. $\lambda (t) \equiv \lambda_0$ is constant, the $V_i$ are uniform random variables on $[0,1]^d$.

Therefore, computing the conditional expectation, we have for $k\geq 1$ 
\begin{align}
    &\E \left[ \left| \sum_{j=1}^{N_s(t)} A_j w \left( \frac{t-t_j}{s} \right) e^{i \xi \cdot (t-t_j)} \right|^p : N_s(t) = k \right]\\ =& \E \left[ \left| \sum_{j=1}^k A_j w (V_j) e^{i s \xi \cdot V_j} \right|^p \right] \nonumber \\
    \leq& \frac{\| \lambda \|_{\infty}}{\lambda_{\min}} k^p \E [|A_1|^p] \| w \|_p^p \, , \label{eqn: proof A1}
\end{align}
where \eqref{eqn: proof A1} follows from (i) the independence of the random variables $A_j$ and $V_j$; (ii) the fact that for any sequence of i.i.d. random variables $Z_1, Z_2, \ldots$,
\begin{equation*}
    \E \left[ \left| \sum_{n=1}^k Z_n \right|^p \right] \leq k^{p-1} \E \left[ \sum_{n=1}^k |Z_n|^p \right] = k^p \E [ |Z_1|^p ] \, ;
\end{equation*}
and (iii) the fact that
\begin{equation*}
    \E [ |w(V_i)|^p ] = \int_{[0,1]^d} |w (v)|^p p_V (v) \, dv \leq \frac{\| \lambda \|_{\infty}}{\lambda_{\min}} \| w \|_p^p \, .
\end{equation*}
Therefore, since $\mathbb{P} [ N_s(t) = k ] = e^{-\Lambda_s(t)} \cdot \nicefrac{(\Lambda_s(t))^k}{k!}$,
\begin{align*}
    \E &\left[ \left| \sum_{j=1}^{N_s(t)} A_j w \left( \frac{t-t_j}{s} \right) e^{i \xi \cdot (t-t_j)} \right|^p \right] = \\
    =& \sum_{k=0}^{\infty} e^{-\Lambda_s(t)} \frac{(\Lambda_s(t))^k}{k!} \E \left[ \left| \sum_{j=1}^{N_s(t)} A_j w \left( \frac{t-t_j}{s} \right) e^{i \xi \cdot (t - t_j)} \right|^p : N_s(t) = k \right] \\
    =& \sum_{k=1}^{\infty} e^{-\Lambda_s(t)} \frac{(\Lambda_s(t))^k}{k!} \E \left[ \left| \sum_{j=1}^k A_j w (V_j) e^{i s \xi \cdot V_j} \right|^p \right] \\
    =& \sum_{k=1}^m e^{-\Lambda_s(t)} \frac{(\Lambda_s(t))^k}{k!} \E \left[ \left| \sum_{j=1}^k A_j w (V_j) e^{i s \xi \cdot V_j} \right|^p \right] + \epsilon (m, s, \xi, t) \, ,
\end{align*}
where
\begin{equation*}
    \epsilon (m, s, t, \xi) := \sum_{k=m+1}^{\infty} e^{-\Lambda_s(t)} \frac{(\Lambda_s(t))^k}{k!} \E \left[ \left| \sum_{j=1}^k A_j w (V_j) e^{i s \xi \cdot V_j} \right|^p \right] \, .
\end{equation*}

By \eqref{eqn: proof A1} and Lemma \ref{lem: poisson variable exp tail}, if $s$ is small enough so that $\Lambda_s(t) \leq s^d \| \lambda \|_{\infty} < 1$, then:
\begin{align*}
    \epsilon (m, s, \xi, t) &= \sum_{k=m+1}^{\infty} e^{-\Lambda_s(t)} \frac{(\Lambda_s(t))^k}{k!} \E \left[ \left| \sum_{j=1}^k A_j w (V_j) e^{i s \xi \cdot V_j} \right|^p \right] \\
    &\leq \frac{\| \lambda \|_{\infty}}{\lambda_{\min}} \E [ |A_1|^p ] \| w \|_p^p \sum_{k=m+1}^{\infty} e^{-\Lambda_s(t)} \frac{ (\Lambda_s(t))^k}{k!} k^p \\
    &\leq C_{m,p} \frac{\| \lambda \|_{\infty}}{\lambda_{\min}} \E [ |A_1|^p ] \| w \|_p^p (\Lambda_s(t))^{m+1} \\
    &\leq C_{m,p} \frac{\| \lambda \|_{\infty}}{\lambda_{\min}} \E [ |A_1|^p ] \| w \|_p^p \| \lambda \|_{\infty}^{m+1} s^{d(m+1)} \, .
\end{align*}
\end{proof}

\section{Proof of Theorem  2}

\begin{proof}
Let $(s_k, \xi_k)$ be a sequence of scale and frequency pairs such that $\lim_{k \rightarrow \infty} s_k = 0$. Applying Theorem \ref{thm: scattering 1st order taylor} with $m = 1$, we obtain:
\begin{align*}
    &\frac{S_{\gamma_k, p} Y (t)}{s_k^d}\\ =& e^{-\Lambda_{s_k} (t)} \frac{ \Lambda_{s_k} (t)}{s_k^d} \E \left[ \left| A_1 w(V_{1,k}) e^{i s \xi \cdot V_{1,k}} \right|^p \right] + \frac{\epsilon (1, s_k, \xi_k, t)}{s_k^d}\\
    =& e^{-\Lambda_{s_k} (t)} \frac{ \Lambda_{s_k} (t)}{s_k^d} \E [ |A_1|^p ] \E [ |w(V_{1,k})|^p ] + \frac{\epsilon (1, s_k, \xi_k, t)}{s_k^d} \, ,
\end{align*}
where we write $V_{1,k} = V_1$ to emphasize the fact that the density of $V_{1,k}$ is:
\begin{equation*}
    p_{V_k} (v) = \frac{s_k^d}{\Lambda_{s_k}(t)} \lambda (t - vs_k) \, .
\end{equation*}
Using the error bound \eqref{eqn: error bound}, we see that:
\begin{equation*}
    \lim_{k \rightarrow \infty} \frac{ \epsilon (1, s_k, \xi_k, t)}{s_k^d} = 0 \, .
\end{equation*}
Furthermore, since $0 \leq \Lambda_{s_k} (t) \leq s_k^d \| \lambda \|_{\infty}$, we observe that:
\begin{equation*}
    \lim_{k \rightarrow \infty} e^{-\Lambda_{s_k}(t)} = 1 \, ,
\end{equation*}
and by the continuity of $\lambda (t)$, 
\begin{equation} \label{eqn: proof B1}
    \lim_{k \rightarrow \infty} \frac{\Lambda_{s_k}(t)}{s_k^d} = \lim_{k \rightarrow \infty} \frac{1}{s_k^d} \int_{[s_k-t, t]^d} \lambda (u) \, du = \lambda (t) \, .
\end{equation}
Finally, by the continuity of $\lambda (t)$, we see that
\begin{equation} \label{eqn: props of pV}
    p_{V_k} (v) \leq \frac{ \| \lambda \|_{\infty} }{\lambda_{\min}} \quad \text{and} \quad \lim_{k \rightarrow \infty} p_{V_k} (v) = 1 \, , \quad \forall \, v \in [0,1]^d \, .
\end{equation}
Therefore, by the bounded convergence theorem, 
\begin{align*}
    \lim_{k \rightarrow \infty} \E [ |w(V_1)|^p] &= \lim_{k \rightarrow \infty} \int_{[0,1]^d} |w(v)|^p p_{V_k}(v) \, dv \\&= \int_{[0,1]^d} |w(v)|^p \lim_{k \rightarrow \infty} p_{V_k} (v) \, dv \\&= \| w \|_p^p \, .
\end{align*}
That completes the proof of \eqref{eqn: location scattering small scales}. 

To prove \eqref{eqn: invariant small scales}, we assume that $\lambda (t)$ is periodic with period $T$ along each coordinate and again use Theorem \ref{thm: scattering 1st order taylor} with $m = 1$ to observe,
\begin{align*}
    &\frac{S Y (s_k, \xi_k, p)}{s_k^d}\\ &= \E [|A_1|^p] \frac{1}{T^d} \int_{[0,T]^d} e^{-\Lambda_{s_k}(t)} \frac{\Lambda_{s_k}(t)}{s_k^d} \times \\
    &\quad \int_{[0,1]^d} |w(v)|^p p_{V_k} (v) \, dv \, dt + \frac{1}{T^d} \int_{[0,1]^d} \frac{\epsilon (1, s_k, \xi_k, t)}{s_k^d} \, dt \, .
\end{align*}
By \eqref{eqn: error bound}, the second integral converges to zero as $k \rightarrow \infty$. Therefore,
\begin{equation*}
    \lim_{k \rightarrow \infty} \frac{S Y (s_k, \xi_k, p)}{s_k^d} = \E [|A_1|^p] \| w \|_p^p \frac{1}{T^d} \int_{[0,T]^d} \lambda (t) \, dt \, ,
\end{equation*}
by the continuity of $\lambda (t)$ and the bounded convergence theorem. 
\end{proof}

\section{Proof of Theorem 3}

\begin{proof} We apply Theorem \ref{thm: scattering 1st order taylor} with $m = 2$ and obtain:
\begin{align}
    &S_{\gamma_k, p} Y (t)\\ =& e^{-\Lambda_{s_k}(t)} \Lambda_{s_k}(t) \E [|A_1|^p] \E[|w(V_{1,k})|^p] \label{eqn: proof C1} \\
    +& e^{-\Lambda_{s_k}(t)} \frac{(\Lambda_{s_k}(t))^2}{2} \E \left[ \left| A_1 w (V_{1,k}) e^{i s_k \xi_k \cdot V_{1,k}} + A_2 w (V_{2,k}) e^{i s_k \xi_k \cdot V_{2,k}} \right|^p \right] \\+& \epsilon (2, s_k, \xi_k, t) \, , \nonumber
\end{align}
where $V_{i,k}$, $i = 1,2$, are random variables taking values on the unit cube $ [0,1]^d$ with densities,
\begin{equation*}
    p_{V_k} (v) = \frac{s_k^d}{\Lambda_{s_k}(t)} \lambda (t - vs_k) \, .
\end{equation*}
Dividing both sides in \eqref{eqn: proof C1} by $s_k^{2d} \| w \|_p^p \E [|A_1|^p]$ and subtracting $\frac{\Lambda_{s_k}(t)}{s_k^{2d}} \frac{\E [|w (V_{1,k})|^p]}{\| w \|_p^p}$ yields:
\begin{align}
    &\frac{S_{\gamma_k, p} Y (t)}{s_k^{2d} \| w \|_p^p \E [|A_1|^p]} - \frac{ \Lambda_{s_k} (t)}{s_k^{2d}} \frac{\E [ |w(V_{1,k})|^p]}{\| w \|_p^p}\\ =& \frac{e^{-\Lambda_{s_k}(t)} \Lambda_{s_k}(t) - \Lambda_{s_k}(t)}{s_k^{2d}} \frac{\E [ |w(V_{1,k})|^p]}{\| w \|_p^p} \label{eqn: proof C2} \\
    +& e^{-\Lambda_{s_k}(t)} \frac{(\Lambda_{s_k}(t))^2}{s_k^{2d}} \frac{\E \left[ \left| A_1 w(V_{1,k}) e^{i s_k \xi_k \cdot V_{1,k}} + A_2 w(V_{2,k}) e^{i s_k \xi_k \cdot V_{2,k}} \right|^p \right]}{2 \| w \|_p^p \E [|A_1|^p]} \nonumber \\
    &\quad\quad+ \frac{\epsilon (2, s_k, \xi_k, t)}{s_k^{2d} \| w \|_p^p \E [|A_1|^p]} \, . \nonumber 
\end{align}
Using the error bound \eqref{eqn: error bound}, 
\begin{equation} \label{eqn: proof C5}
    \lim_{k \rightarrow \infty} \frac{\epsilon (2, s_k, \xi_k, t)}{s_k^{2d} \| w \|_p^p \E [|A_1|^p]} = 0 \, ,
\end{equation}
at a rate independent of $t$. Recalling \eqref{eqn: props of pV} from the proof of Theorem \ref{thm: scattering 1st order small scale limit}, we use the fact that $\lim_{k \rightarrow \infty} p_{V_k} \equiv 1$ and the bounded convergence theorem to conclude, 
\begin{align} \label{eqn: proof C4}
    &\lim_{k \rightarrow \infty} \E \left[ \left| A_1 w(V_{1,k}) e^{i s_k \xi_k \cdot V_{1,k}} + A_2 w(V_{2,k}) e^{i s_k \xi_k \cdot V_{2,k}} \right|^p \right]\\ &= \E \left[ \left| A_1 w (U_1) e^{i L \cdot U_1} + A_2 w(U_2) e^{i L \cdot U_2} \right|^p \right] \, ,
\end{align}
where $U_i$, $i=1,2$, are uniform random variables on the unit cube and $L = \lim_{k \rightarrow \infty} s_k \xi_k$. Similarly,
\begin{equation} \label{eqn: V1k limit}
    \lim_{k \rightarrow \infty} \frac{\E [|w(V_{1,k})|^p]}{\| w \|_p^p} = 1 \, .
\end{equation}
Lastly, recalling that $s_k \rightarrow 0$ as $k \rightarrow \infty$ and using \eqref{eqn: proof B1} from the proof of Theorem \ref{thm: scattering 1st order small scale limit}, we see
\begin{align}
    &\lim_{k \rightarrow \infty} \frac{e^{-\Lambda_{s_k}(t)} \Lambda_{s_k}(t) - \Lambda_{s_k}(t)}{s_k^{2d}} \\=& \lim_{k \rightarrow \infty} \left( \frac{\Lambda_{s_k}(t)}{s_k^d} \right) \lim_{k \rightarrow \infty} \left( \frac{e^{-\Lambda_{s_k}(t)} - 1}{s_k^d} \right) \nonumber \\
    =& \lambda (t) \lim_{k \rightarrow \infty} \left( \frac{e^{-\Lambda_{s_k}(t)} - 1}{s_k^d} \right) \nonumber \\
    =& -\lambda (t)^2 \, . \label{eqn: proof C3}
\end{align}

Now we integrate both sides of \eqref{eqn: proof C2} over $[0,T]^d$ and divide by $T^d$. Taking the limit as $k \rightarrow \infty$, on the left hand side we get:
\begin{align*}
    &\hskip -10pt \lim_{k \rightarrow \infty} \frac{1}{T^d} \int_{[0,T]^d} \left( \frac{S_{\gamma_k, p} Y (t)}{s_k^{2d} \| w \|_p^p \E [|A_1|^p]} - \frac{ \Lambda_{s_k} (t)}{s_k^{2d}} \frac{\E [ |w(V_{1,k})|^p]}{\| w \|_p^p} \right) dt \\
    &= \lim_{k \rightarrow \infty} \left( \frac{SY (s_k, \xi_k, p)}{s_k^{2d} \| w \|_p^p \E [|A_1|^p]} - \frac{ \E [|w (V_{1,k})|^p]}{\| w \|_p^pT^d} \int_{[0,T]^d} \frac{\Lambda_{s_k}(t)}{s_k^{2d}} \, dt \right) \\
    &= \lim_{k \rightarrow \infty} \left( \frac{SY (s_k, \xi_k, p)}{s_k^{2d} \E [|w (V_{1,k})|^p] \E [|A_1|^p]} - \frac{1}{T^d} \int_{[0,T]^d} \frac{\Lambda_{s_k}(t)}{s_k^{2d}} \, dt \right) \, ,
\end{align*}
where we used the definition of the invariant scattering moments and \eqref{eqn: V1k limit}. On the right hand side of \eqref{eqn: proof C2}, we use \eqref{eqn: V1k limit}, \eqref{eqn: proof C3} and the dominated convergence theorem to see that the first term is:
\begin{align*}
    &\lim_{k \rightarrow \infty} \frac{1}{T^d} \int_{[0,T]^d} \frac{e^{-\Lambda_{s_k}(t)} \Lambda_{s_k}(t) - \Lambda_{s_k}(t)}{s_k^{2d}} \frac{\E [ |w(V_{1,k})|^p]}{\| w \|_p^p} \, dt \\&= \lim_{k \rightarrow \infty} \frac{1}{T^d} \int_{[0,T]^d} \frac{e^{-\Lambda_{s_k}(t)} \Lambda_{s_k}(t) - \Lambda_{s_k}(t)}{s_k^{2d}} \, dt \\
    &= -\frac{1}{T^d} \int_{[0,T]^d} \lambda (t)^2 \, dt \, .
\end{align*}
Using \eqref{eqn: proof B1}, \eqref{eqn: proof C4}, and the bounded convergence theorem, the second term of \eqref{eqn: proof C2} is:
\begin{align*}
    &\lim_{k \rightarrow \infty} \frac{1}{T^d} \int_{[0,T]^d} e^{-\Lambda_{s_k}(t)} X_k \, dt \\
    &= \frac{\E [ |A_1 w (U_1) e^{i L \cdot U_1} + A_2 w (U_2) e^{i L \cdot U_2} |^p ]}{2 T^d\| w \|_p^p \E [|A_1|^p] } \left(  \int_{[0,T]^d} \lambda(t)^2 \, dt \right)
\end{align*}
where 
\begin{equation*}
    X_k=\frac{\E \left[ \left| A_1 w(V_{1,k}) e^{i s_k \xi_k \cdot V_{1,k}} + A_2 w(V_{2,k}) e^{i s_k \xi_k \cdot V_{2,k}} \right|^p \right]}{2 \| w \|_p^p \E [|A_1|^p]}.
\end{equation*}
Finally, the third term of \eqref{eqn: proof C2} goes to zero using the bounded convergence theorem and \eqref{eqn: proof C5}. Putting together the left and right hand sides of \eqref{eqn: proof C2} with these calculations finishes the proof.
\end{proof}

\section{Proof of Theorem 4}

\begin{proof} As in the proof of Theorem \ref{thm: scattering 1st order taylor}, let $N_s(t) = N\left([t-s, t]^d\right)$ denote the number of points in the cube $[t-s,t]^d$. Then since the support of $w$ is contained in $[0,1]^d$,
\begin{align*}
    \left(g_{\gamma_k} \ast Y\right) (t) &= \int_{[t-s_k, t]^d} w \left( \frac{t-u}{s_k} \right) e^{i \xi_k \cdot (t-u)} \, Y (du)\\ &= \sum_{j=1}^{N_{s_k}(t)} A_j w \left ( \frac{t - t_j}{s_k} \right) e^{i \xi_k \cdot (t - t_j)} \, ,
\end{align*}
where $t_1, t_2, \ldots, t_{N_{s_k}(t)}$ are the points of $N$ in $[t - s_k, t]^d$. Therefore, in the event that $N_{s_k}(t) = 1$, 
\begin{equation*}
    | \left(g_{\gamma_k} \ast Y\right) (t)|^p = \left(|g_{\gamma_k}|^p \ast |Y|^p\right) (t) \, ,
\end{equation*}
and so, partitioning the space of possible outcomes based on $N_{s_k}(t)$, we obtain:
\begin{align*}
    &| \left(g_{\gamma_k} \ast Y\right) (t)|^p\\ =& | \left(g_{\gamma_k} \ast Y\right) (t) \cdot \mathbbm{1}_{\{ N_{s_k}(t) = 1 \}} + \left(g_{\gamma_k} \ast Y\right)(t) \cdot \mathbbm{1}_{\{ N_{s_k}(t) > 1 \}} |^p \\
    =& | \left(g_{\gamma_k} \ast Y\right) (t) \cdot \mathbbm{1}_{\{ N_{s_k}(t) = 1 \}}|^p + |\left(g_{\gamma_k} \ast Y\right)(t) \cdot \mathbbm{1}_{\{ N_{s_k}(t) > 1 \}} |^p \\
    =&  \left(|g_{\gamma_k}|^p \ast |Y|^p\right) (t) \cdot \mathbbm{1}_{\{ N_{s_k}(t) = 1 \}} + |\left(g_{\gamma_k} \ast Y\right)(t) \cdot \mathbbm{1}_{\{ N_{s_k}(t) > 1 \}} |^p \\
    =&  \left(|g_{\gamma_k}|^p \ast |Y|^p\right) (t) + e_k (t) \, ,
\end{align*}
where
\begin{equation*}
    e_k (t) := |\left(g_{\gamma_k} \ast Y\right)(t) \cdot \mathbbm{1}_{\{ N_{s_k}(t) > 1 \}} |^p -  \left(|g_{\gamma_k}|^p \ast |Y|^p\right) (t) \cdot \mathbbm{1}_{\{ N_{s_k}(t) > 1 \}}
\end{equation*}
Using the above, we can write the second order convolution term as:
\begin{equation*}
    \left(g_{\gamma_k'} \ast |g_{\gamma_k} \ast Y|\right)(t) = \left(g_{\gamma_k'} \ast |g_{\gamma_k}|^p \ast |Y|^p\right) (t) + \left(g_{\gamma_k'} \ast e_k\right) (t) \, .
\end{equation*}
The following lemma implies that $\left(g_{\gamma_k'} \ast e_k\right) (t)$ decays rapidly in $\Lb^{p'}$ at a rate independent of $t$.

\begin{lemma} \label{lem: 2nd order lemma 1}
There exists $\delta > 0$, independent of $t$, such that if $s_k < \delta$, 
\begin{equation*}
    \E \left[ \left|\left(g_{\gamma_k'} \ast e_k\right) (t)\right|^p \right] \leq C (p, p', w, c, L) \frac{ \| \lambda \|_{\infty}}{\lambda_{\min}} \| \lambda \|_{\infty}^2 s_k^{d(p' + 2)} \, .
\end{equation*}
\end{lemma}

Once we have proved Lemma \ref{lem: 2nd order lemma 1}, equation \eqref{eqn: 2nd order loc dep asymp} will follow once we show,
\begin{equation} \label{eqn: proof D1}
    \lim_{k \rightarrow \infty} \frac{\E \left[ \left|\left(g_{\gamma_k'} \ast |g_{\gamma_k}|^p \ast |Y|^p\right) (t)\right|^{p'} \right]}{s_k^{d(p' + 1)}} = K(p, p', w, c, L) \lambda (t) \E [|A_1|^q] \, .
\end{equation}

Let us prove \eqref{eqn: proof D1} first and postpone the proof of Lemma \ref{lem: 2nd order lemma 1}. We will use the fact that the support of $g_{\gamma_k'} \ast |g_{\gamma_k}|^p$ is contained in $[0, s_k + s_k']^d$. Let $\tilde{s}_k := s_k + s_k'$, $N_k(t) := N_{\tilde{s}_k}(t)$, $\Lambda_k(t) := \Lambda_{\tilde{s}_k}(t)$, and let $t_1, t_2, \ldots, t_{N_k(t)}$ be the points of $N$ in the cube $[t - \tilde{s}_k, t]^d$. We have that $\mathbb{P} [ N_k(t) = n ] = e^{-\Lambda_k(t)} \frac{(\Lambda_k (t))^n}{n!}$, and conditioned on the event that $N_k (t) = n$, the locations of the points $t_1, \ldots, t_n$ are distributed as i.i.d. random variables $Z_1(t), \ldots, Z_n(t)$ taking values in $[t - \tilde{s}_k, t]^d$ with density $p_{Z(t)} (z) = \frac{\lambda (z)}{\Lambda_k (t)}$. Therefore the i.i.d. random variables $\widetilde{V}_1(t), \ldots, \widetilde{V}_n(t)$ defined by $\widetilde{V}_i(t) := t - Z_i(t)$ take values in $[0, \tilde{s}_k]^d$ and have density
\begin{equation*}
    p_{\widetilde{V}(t)} (v) = \frac{\lambda (t-v)}{\Lambda_k(t)} \, , \quad v \in [0, \tilde{s}_k]^d \, .
\end{equation*}
Now, we condition on $N_k(t)$ to see that
\begin{align}
    &\E \left[ \left|\left(g_{\gamma_k'} \ast |g_{\gamma_k}|^p \ast |Y|^p\right) (t)\right|^{p'} \right]\\ =& \E \left[ \left| \sum_{j=1}^{N_k(t)} |A_j|^p \left(g_{\gamma_k'} \ast |g_{\gamma_k}|^p\right) (t - t_j) \right|^{p'} \right] \nonumber \\
    =& \sum_{n=1}^{\infty} e^{-\Lambda_k(t)} \frac{(\Lambda_k(t))^n}{n!} \\
    &\quad\quad\cdot\E \left[ \left| \sum_{j=1}^{N_k(t)} |A_j|^p \left(g_{\gamma_k'} \ast |g_{\gamma_k}|^p\right) (t-t_j) \right|^{p'} : N_k(t) = n \right] \nonumber \\
    =& \sum_{n=1}^{\infty} e^{-\Lambda_k(t)} \frac{(\Lambda_k(t))^n}{n!} \E \left[ \left| \sum_{j=1}^{n} |A_j|^p \left(g_{\gamma_k'} \ast |g_{\gamma_k}|^p\right) (\widetilde{V}_j(t)) \right|^{p'} \right] \nonumber \\
    =& e^{-\Lambda_k(t)} \Lambda_k(t) \E[|A_1|^q] \E \left[ \left|\left( g_{\gamma_k'} \ast |g_{\gamma_k}|^p\right) (\widetilde{V}_1(t))\right|^{p'} \right] \label{eqn: proof D2} \\
    +& \sum_{n=2}^{\infty} e^{-\Lambda_k(t)} \frac{(\Lambda_k(t))^n}{n!} \E \left[ \left| \sum_{j=1}^{n} |A_j|^p \left(g_{\gamma_k'} \ast |g_{\gamma_k}|^p\right) (\widetilde{V}_j(t)) \right|^{p'} \right] \label{eqn: proof D5} 
\end{align}
The following lemma will be used to estimate the scaling of the term in \eqref{eqn: proof D2}.

\begin{lemma} \label{lem: 2nd order lemma 2}
For all $t \in \R^d$, 
\begin{equation} \label{eqn: 2 layer single pt asymp}
    \hskip -10pt \lim_{k \rightarrow \infty} \frac{ \tilde{s}_k^d}{s_k^{d(p'+1)}} \E \left [ \left| \left(g_{\gamma_k'} \ast |g_{\gamma_k}|^p \right)(\widetilde{V}_1(t)) \right|^{p'} \right] = \| g_{c,L/c} \ast |g_{1,0}|^p \|_{p'}^{p'} \, .
\end{equation}
Furthermore, there exists $\delta > 0$, independent of $t$, such that if $s_k < \delta$ then
\begin{equation} \label{eqn: 2 layer single pt bound}
    \hskip -10pt \frac{\tilde{s}_k^d}{s_k^{d(p'+1)}} \E \left[ \left|\left(g_{\gamma_k'} \ast |g_{\gamma_k}|^p\right) (\widetilde{V}_1(t))\right|^{p'} \right] \leq 2 \frac{\| \lambda \|_{\infty}}{\lambda_{\min}} C(p, p', w, c, L) \, .
\end{equation}
\end{lemma}

\begin{proof}
Making a change of variables in both $u$ and $v$, and recalling the assumption that $s_k' = cs_k$, we observe that
\begin{align}
    &\frac{\tilde{s}_k^d}{s_k^{d(p'+1)}} \E \left[ \left|\left(g_{\gamma_k'} \ast |g_{\gamma_k}|^p\right) (\widetilde{V}_1(t))\right|^{p'} \right] \nonumber \\
    &= \frac{\tilde{s}_k^d}{s_k^{d(p'+1)}} \int_{\R^d} p_{\widetilde{V}(t)} (v)\cdot \\&\quad\quad\quad\quad\quad\quad\quad\quad\quad\left| \int_{\R^d} w \left( \frac{v-u}{s_k'} \right) e^{i \xi_k' \cdot (v-u)} \left| w \left( \frac{u}{s_k} \right) \right|^p \, du \right|^{p'} \, dv \nonumber \\
    &= \tilde{s}_k^d \int_{\R^d} p_{\widetilde{V}(t)} (s_k v) \left| \int_{\R^d} w \left( \frac{s_k (v-u)}{s_k'} \right) e^{i s_k \xi_k' \cdot (v-u)} |w(u)|^p \, du \right|^{p'} \, dv \nonumber \\
    &= \int_{\R^d} \frac{\tilde{s}_k^d \lambda (t - s_k v)}{\Lambda_k(t)} \left| \int_{\R^d} w \left( \frac{u-v}{c} \right) e^{i s_k' \xi_k' \cdot (u-v)/c} |w(u)|^p \, du \right|^{p'} \, dv \, . \label{eqn: proof D3}
\end{align}
The continuity of $\lambda (t)$ implies that
\begin{equation*}
    \lim_{k \rightarrow \infty} \frac{\tilde{s}_k^d \lambda (t-s_k v)}{\Lambda_k(t)} = 1 \, , \quad \forall \, v \in [0, 1+c]^d \, .
\end{equation*}
Furthermore, the assumption $0 < \lambda_{\min} \leq \| \lambda \|_{\infty} < \infty$ implies
\begin{equation} \label{eqn: proof D4}
    \frac{\tilde{s}_k^d \lambda (t-s_k v)}{\Lambda_k(t)} \leq \frac{\| \lambda \|_{\infty}}{\lambda_{\min}} \, , \quad \forall \, k \geq 1 \, .
\end{equation}
Therefore, \eqref{eqn: 2 layer single pt asymp} follows from the dominated convergence theorem and by the observation that the inner integral of \eqref{eqn: proof D3} is zero unless $v \in [0, 1+c]^d$. Equation \eqref{eqn: 2 layer single pt bound} follows from inserting \eqref{eqn: proof D4} into \eqref{eqn: proof D3} and sending $k$ to infinity.
\end{proof}

Since
\begin{equation*}
    \lim_{k \rightarrow \infty} \frac{\Lambda_k(t)}{\tilde{s}_k^d} = \lambda (t) \, ,
\end{equation*}
the independence of $\widetilde{V}_1(t)$ and $A_1$, the continuity of $\lambda (t)$, and Lemma \ref{lem: 2nd order lemma 2} imply that taking $k \rightarrow \infty$ in \eqref{eqn: proof D2} yields:
\begin{align*}
    &\lim_{k \rightarrow \infty} \left( \frac{e^{-\Lambda_k(t)} \Lambda_k(t) \E[|A_1|^q] \E \left[ | g_{\gamma_k'} \ast |g_{\gamma_k}|^p (\widetilde{V}_1(t))|^{p'} \right]}{s_k^{d(p'+1)}} \right) \\
    &= \lim_{k \rightarrow \infty} \left( e^{-\Lambda_k(t)} \frac{\Lambda_k(t)}{\tilde{s}_k^d} \E[|A_1|^q] \frac{\tilde{s}_k^d}{s_k^{d(p'+1)}} \E \left[ | g_{\gamma_k'} \ast |g_{\gamma_k}|^p (\widetilde{V}_1(t))|^{p'} \right]  \right) \\
    &= K(p, p', c, w, L) \lambda (t) \E [|A_1|^q] \, .
\end{align*}

The following lemma shows that \eqref{eqn: proof D5} is $O\left(s_k^{d(p'+2)}\right)$ (and converges at a rate independent of $t$), and therefore completes the proof of \eqref{eqn: 2nd order loc dep asymp} subject to proving Lemma \ref{lem: 2nd order lemma 1}. 

\begin{lemma}\label{ 2nd order lemma 3}
For all $\alpha \in \R$ there exists $\delta > 0$, independent of $t$, such that if $s_k < \delta$, then
\begin{align*}
    &\sum_{n=2}^{\infty} e^{-\Lambda_k(t)} \frac{(\Lambda_k(t))^n}{n!} n^{\alpha} \E \left[ \left| \sum_{j=1}^n |A_j|^p \left(g_{\gamma_k'} \ast |g_{\gamma_k}|^p\right) (\widetilde{V}_j(t)) \right|^{p'} \right] \\
    &\leq C(p, p', w, c, \alpha, L) \frac{\| \lambda \|_{\infty}}{\lambda_{\min}} \| \lambda \|_{\infty}^2 \E [|A_1|^q] s_k^{d(p'+2)} \, .
\end{align*}
\end{lemma}

\begin{proof}
For any sequence of i.i.d. random variables, $Z_1,Z_2,\ldots,$ it holds that 
\begin{equation*}
\mathbb{E}\left[\left|\sum_{n=1}^k Z_n\right|^p\right] \leq k^{p-1} \mathbb{E}\left[\sum_{n=1}^k |Z_n|^p\right]=k^p\mathbb{E}\left[|Z_1|^p\right].
\end{equation*} 
Therefore, by
Lemma \ref{lem: poisson variable exp tail}, Lemma \ref{lem: 2nd order lemma 2},  
and the fact that the $\widetilde{V}_j(t)$ and $A_i$ are i.i.d. and independent of each other, we see that if $s_{k}<\delta,$ where $\delta$ is as in (\ref{eqn: 2 layer single pt bound}),
\begin{align*}
&\sum_{n=2}^\infty e^{-\Lambda_k(t)}\frac{(\Lambda_k(t))^n}{n!}n^{\alpha}\\
&\quad\quad\quad\quad\quad\quad\quad\times\mathbb{E}\left[\left|\sum_{j=1}^n|A_i|^{p}\left(g_{\gamma_k'}\ast\left|g_{\gamma_k}\right|^{p}\right)(\widetilde{V}_j(t))\right|^{p'}\right]\\
\leq&\sum_{n=2}^\infty e^{-\Lambda_k(t)}\frac{(\Lambda_k(t))^n}{n!}n^{\alpha}n^{p'}\\
&\quad\quad\quad\quad\quad\quad\quad\times\mathbb{E}\left[|A_1|^{q}\left|\left(g_{\gamma_k'}\ast\left|g_{\gamma_k}\right|^{p}\right)(\widetilde{V}_1(t))\right|^{p'}\right]\\
=&\sum_{n=2}^\infty e^{-\Lambda_k(t)}\frac{(\Lambda_k(t))^n}{n!} n^{p'+\alpha}\\
&\quad\quad\quad\quad\quad\quad\quad\times\mathbb{E}[|A_1|^{q}]\mathbb{E}\left[\left|\left(g_{\gamma_k'}\ast\left|g_{\gamma_k}\right|^{p}\right)(\widetilde{V}_1(t))\right|^{p'}\right]\\
=&\mathbb{E}[|A_1|^{q}]\mathbb{E}\left[\left|\left(g_{\gamma_k'} \ast\left|g_{\gamma_k}\right|^{p}\right)(\widetilde{V}_1(t))\right|^{p'}\right]\\&
\quad\quad\quad\quad\quad\quad\quad\times\sum_{n=2}^\infty e^{-\Lambda_k(t)}\frac{(\Lambda_k(t))^{n}}{n!}n^{p'+\alpha}\\
\leq&C (p,p',w,c,L)\frac{\|\lambda\|_\infty}{\lambda_{\min}}\mathbb{E}[|A_1|^{q}]\frac{s_{k}^{d(p'+1)}}{\tilde{s}_k^d}\\
&\quad\quad\quad\quad\quad\quad\times\sum_{n=2}^\infty e^{-\Lambda_k(t)}\frac{(\Lambda_k(t))^{n}}{n!} n^{p'+\alpha}\\
\leq&C (p,p',w,c,L,\alpha) \frac{\|\lambda\|_\infty}{\lambda_{\min}}\mathbb{E}[|A_1|^{q}]\frac{s_{k}^{d(p'+1)}}{\tilde{s}_k^d}(\Lambda_k(t))^2\\
\leq& C (p,p',w,c,L,\alpha) \frac{\|\lambda\|_\infty}{\lambda_{\min}}\|\lambda\|_\infty^2\mathbb{E}[|A_1|^{q}]s_{k}^{d(p'+2)},
\end{align*}
where the last inequality uses the fact that $\Lambda_k (t) \leq \tilde{s}_k^d \|\lambda\|_\infty =(1+c)^ds_{k}^d \|\lambda\|_\infty.$
\end{proof}

We will now complete the proof of the theorem by proving Lemma \ref{lem: 2nd order lemma 1}.
\begin{proof}\textbf{[Lemma \ref{lem: 2nd order lemma 1}]}
Since 
\begin{align*}
&e_k(t)\\
=&|\left(g_{\gamma_k}\ast Y\right)(t)\mathbbm{1}_{\{ N_{s_{k}}(t)>1 \}}|^{p}-\left(\left|g_{\gamma_k}\right|^{p}\ast |Y|^{p}\right)(t)\mathbbm{1}_{\{N_{s_{k}}(t)>1\}},
\end{align*}
we see that
\begin{align*}
\left|g_{\gamma_k'}\ast e_k (t)\right| 
&\leq \left|g_{\gamma_k'}\ast\left(\left|\left(g_{\gamma_k}\ast Y\right)\mathbbm{1}_{\{N_{s_{k}}(\cdot)>1\}}\right|^{p}\right)(t)\right|\\&\quad\quad\quad+ \left|g_{\gamma_k'}\ast\left(\left(|g_{\gamma_k}|^{p}\ast |Y|^{p}\right)\mathbbm{1}_{\{N_{s_{k}}(\cdot)>1\}}\right)(t)\right|.
\end{align*}
First turning our attention to the second term, we note that 
\begin{align}
&\Big|g_{\gamma_k'} \ast \left(\left(\left|g_{\gamma_k}\right|^{p}\ast |Y|^{p}\right)\mathbbm{1}_{\{N_{s_{k}}(\cdot)>1\}}\right)(t)\Big| \nonumber \\
&=\left|\int_{[t-s'_{k},t]^d}w\left(\frac{t-u}{s'_{k}}\right)e^{i\xi_{k}'\cdot(t-u)}\left(\left|g_{\gamma_k}\right|^{p}\ast |Y|^{p}\right)(u)\mathbbm{1}_{\{N_{s_{k}}(u)>1\}} \, du\right|\nonumber\\
&\leq \mathbbm{1}_{\{N_k(t)>1\}}\int_{[t-s'_{k},t]^d}w\left(\frac{t-u}{s'_{k}}\right)\left(\left|g_{\gamma_k}\right|^{p}\ast |Y|^{p}\right)(u) \, du\nonumber\\
&= \mathbbm{1}_{\{N_k(t)>1\}}\left(g_{s'_{k},0}\ast\left|g_{\gamma_k}\right|^{p}\ast |Y|^{p}\right)(t).\label{subint}
\end{align}
since $N_{s_{k}}(u)\leq N_{s_k + s_k'}(t) = N_{\tilde{s}_k}(t) = N_k(t)$ 
 for all $u\in[t-s'_{k},t]^d.$
Therefore, conditioning on $N_k(t),$ if $s_{k}<\delta,$ 
\begin{align*}
\mathbb{E} &\left[\left|g_{\gamma_k'}\ast\left(\left(\left|g_{\gamma_k}\right|^{p}\ast |Y|^{p}\right)\mathbbm{1}_{\{N_{s_{k}}(\cdot)>1\}}\right)(t)\right|^{p'}\right]\\
&\leq \mathbb{E}\left[\left|\mathbbm{1}_{\{N_k(t)>1\}}\left(g_{s'_{k},0}\ast\left|g_{\gamma_k}\right|^{p}\ast |Y|^{p}\right)(t)\right|^{p'}\right] \\
&= \sum_{n=2}^\infty e^{-\Lambda_k(t)}\frac{(\Lambda_k(t))^n}{n!}\mathbb{E}\left[\left|\sum_{j=1}^n|A_j|^{p}\left(g_{s'_{k},0}\ast\left|g_{\gamma_k}\right|^{p}\right)(\widetilde{V}_j(t))\right|^{p'}\right]\\
&\leq C (p,p',w,c,L) \frac{\|\lambda\|_\infty}{\lambda_{\min}} \|\lambda\|_\infty^{2}\mathbb{E}[|A_1|^{q}]s_{k}^{d(p'+2)}
\end{align*}
by Lemma \ref{ 2nd order lemma 3}. 
Now, turning our attention to the first term, note that
\begin{equation*}
\left|(g_{\gamma_k}\ast Y)(t)\right|^{p}\mathbbm{1}_{\{N_{s_{k}}(t)>1 \}}\leq N_{s_{k}}(t)^{p-1}\left(\left|g_{\gamma_k}\right|^{p}\ast|Y|^{p}\right)(t)\mathbbm{1}_{\{ N_{s_{k}}(t)>1 \}} \, .
\end{equation*}
Therefore, by the same logic as in (\ref{subint})
\begin{align*}
 &\Big|g_{\gamma_k'} \ast \left(\left|\left(g_{\gamma_k}\ast Y\right)\mathbbm{1}_{\{N_{s_{k}}(\cdot)>1\}}\right|^{p}\right)(t)\Big| \\
 &\leq\int_{[t-s'_{k},t]^d} w\left(\frac{t-u}{s'_{k}}\right) N_{s_{k}}(u)^{p-1}\left(\left|g_{\gamma_k}\right|^{p}\ast|Y|^{p}\right)(u)\mathbbm{1}_{\{N_{s_{k}}(u)>1\}} \, du\\
&\leq \mathbbm{1}_{\{N_k(t)>1\}} N_k(t)^{p-1}\int_{[t-s'_{k},t]^d} w\left(\frac{t-u}{s_k'}\right) \left(\left|g_{\gamma_k}\right|^{p}\ast|Y|^{p}\right)(u) \, du\\
&\leq \mathbbm{1}_{\{N_k(t)>1\}}N_k(t)^{p-1}\left(g_{s'_{k},0}\ast\left(\left|g_{\gamma_k}\right|^{p}\ast|Y|^{p}\right)\right)(t) \, .
\end{align*}
So again conditioning on $N_k(t),$ and applying Lemma \ref{ 2nd order lemma 3}, we see that if $s_{k}<\delta$
\begin{align*}
 &\mathbb{E} \left[\left|g_{\gamma_k'}\ast\left(\left|\left(g_{\gamma_k}\ast Y\right)\mathbbm{1}_{\{N_{s_{k}}(\cdot)>1\}}\right|^{p}\right)(t)\right|^{p'}\right]\\
&\leq \sum_{n=2}^\infty e^{-\Lambda_k(t)}\frac{(\Lambda_k(t))^k}{n!}n^{p-1}\mathbb{E}\left[\left|\sum_{j=1}^n|A_j|^p \left| \left(g_{s'_{k},0}\ast\left|g_{\gamma_k}\right|\right)(\widetilde{V}_j(t)) \right|^{p}\right|^{p'}\right]\\
&\leq  C (p,p',w,c,L) \frac{\|\lambda\|_\infty}{\lambda_{\min}}\|\lambda\|_\infty^{2}\mathbb{E}[|A_1|^{q}]s_{k}^{d(p'+2)}.
\end{align*}

\end{proof}
This completes the proof of (\ref{eqn: 2nd order loc dep asymp}). Line (\ref{eqn: 2nd order loc dep asymp inv}) follows from integrating with respect to $t,$ observing that the error bounds in Lemmas \ref{lem: 2nd order lemma 1} and \ref{lem: 2nd order lemma 2} are independent of $t,$ and applying the bounded convergence theorem.

\end{proof}

\section{The Proof of Theorem 5} \label{nonpoipf}

In order to prove Theorems \ref{thm: alpha stable}, we will need the following lemma which shows that the scaling relationship of a self-similar process $X (t) $ induces a similar relationship on stochastic integrals against $dX (t).$

\begin{lemma}\label{lem: scalint}
Let $X$ be a stochastic process that satisfies the scaling relation
\begin{equation} \label{eqn: scaling relation}
    X(st) =_d s^{\beta} X(t)
\end{equation}
for some $\beta > 0$ (where $=_d$ denotes equality in distribution). Then for any measurable function $f : \R \rightarrow \R$,
\begin{equation*}
    \int_0^s f(u) \, dX (u) \coloneqq s^{\beta} \int_0^1 f(su) \, dX(u) \, .
\end{equation*}
\end{lemma}

\begin{proof}
Let $X = (X (t))_{t\in\R}$ be a stochastic process satisfying \eqref{eqn: scaling relation}, and let  $\mathcal{P}_n=\{0=t_0^n<t_1^n<\ldots<t_{K_n}^n=1\}$ be a sequence of partitions of $[0,1]$ such that 
\begin{equation*}
\lim_{n\rightarrow\infty} \max_k \left\{ |t^n_k-t^n_{k-1}| \right\} = 0.
\end{equation*}
Then, by the scaling relation \eqref{eqn: scaling relation},
\begin{align*}
&\int_0^s f(u) \, dX (u)\\ =& \lim_{n\rightarrow\infty} \sum_{k=0}^{{K_n}-1} f(st_k^n)\left(X (st_{k+1}^n)-X (st_{k}^n) \right)\\
\coloneqq& s^\beta \lim_{n\rightarrow\infty} \sum_{k=0}^{K_n -1} f(st_k^n) \left(X (t_{k+1}^n) - X (t_{k}^n) \right)\\
=&s^\beta \int_0^1 f(su) \, dX (u) \, .
\end{align*}
\end{proof}

We will now use Lemma \ref{lem: scalint} to prove Theorem \ref{thm: alpha stable}.

\begin{proof} We first consider the case where $X=(X (t))_{t\in\mathbb{R}}$ is an $\alpha$-stable process, $p<\alpha\leq 2.$ Since $X$ has stationary increments, its scattering coefficients do not depend on $t$ and it suffices to analyze
\begin{align*}
\mathbb{E}\left[ \left|(g_{\gamma_k}\ast dX)(0)\right|^p \right] &=\mathbb{E} \left[ \left| \int_{-s_k}^0g_{\gamma_k}(u) \, dX (u) \right|^p \right]\\ &=\mathbb{E} \left[ \left| \int_0^{s_k}g_{\gamma_k}(u) \, dX (u) \right|^p \right] \, ,
\end{align*}
where the second equality uses the fact the distribution of $X$ does not change if it is run in reverse, i.e.
\begin{equation*}
(X (t))_{t\in\mathbb{R}} =_d (X (-t))_{t\in\mathbb{R}}
\end{equation*}
It is well known that  $X (t)$ satisfies \eqref{eqn: scaling relation} for $\beta=\nicefrac{1}{\alpha}.$ Therefore, by Lemma \ref{lem: scalint}
\begin{align*}
\E \left[ \left|(g_{\gamma_k}\ast dX)(0)\right|^p \right] &= \E \left[ \left|\int_0^{s_k} w\left(\frac{u}{s_k}\right)e^{i\xi_k u} \, dX (u) \right|^p \right]\\ &= s_k^{\nicefrac{p}{\alpha}} \E \left[ \left|\int_0^1 w(u)e^{i\xi_k s_k u} \, dX (u) \right|^p \right] \, . 
\end{align*}
So,
\begin{equation*}
\frac{\mathbb{E}\left[ \left|(g_{\gamma_k}\ast dX)(0)\right|^p \right]}{s_k^{\nicefrac{p}{\alpha}}} =  \mathbb{E} \left[ \left|\int_0^1w(u)e^{i\xi_k s_ku} \, dX (u) \right|^p \right] \, .
\end{equation*}
The proof will be complete as soon as we show that 
\begin{align*}
 &\lim_{k\rightarrow\infty}\left(\mathbb{E} \left[ \left|\int_0^1w(u)e^{i\xi_k s_ku} \, dX (u) \right|^p \right] \right)^{\nicefrac{1}{p}}\\=& \left(\mathbb{E} \left[ \left|\int_0^1w(u)e^{iLu} \, dX (u) \right|^p \right] \right)^{\nicefrac{1}{p}} \, .
\end{align*}
By the triangle inequality,
\begin{align*}
\Bigg| &\left( \mathbb{E} \left[ \left|\int_0^1w(u)e^{i\xi_ks_ku} \, dX (u) \right|^p \right] \right)^{\nicefrac{1}{p}}\\&\quad\quad\quad\quad\quad\quad\quad\quad\quad\quad\quad\quad\quad- \left(\mathbb{E} \left[ \left| \int_0^1 w(u) e^{iLu} \, dX (u) \right|^p \right] \right)^{\nicefrac{1}{p}}\Bigg| \\
&\leq \left( \mathbb{E} \left[ \left|\int_0^1w(u)\left(e^{i\xi_k s_ku}-e^{iLu}\right) \, dX (u) \right|^p \right] \right)^{\nicefrac{1}{p}} \, .
\end{align*}
Since $1\leq p<\alpha,$ we may choose $p'$ strictly greater than $1$ such that $p \leq p'< \alpha,$ and note
that by Jensen's inequality
\begin{align*}
&\left(\mathbb{E} \left[ \left|\int_0^1w(u)\left(e^{i\xi_k s_ku}-e^{iLu}\right) \, dX (u) \right|^p \right] \right)^{\nicefrac{1}{p}}\\\leq& \left(\mathbb{E} \left[ \left|\int_0^1w(u)\left(e^{i\xi_k s_ku}-e^{iLu}\right) \, dX (u)\right|^{p'} \right] \right)^{\nicefrac{1}{p'}} \, ,
\end{align*}
and since $X (t)$ is a $p'$-integrable martingale, the boundedness of martingale transforms  (see \cite{Burkholder88}  and  also  \cite{banuelos1995}) implies
\begin{align*}
\Bigg(\mathbb{E} &\left[ \left|\int_0^1w(u)\left(e^{i\xi_k s_ku}- e^{i L u} \right) \,  dX (u) \right|^{p'} \right] \Bigg)^{\nicefrac{1}{p'}} \\ 
&\leq C_{p'} \sup_{0\leq u\leq 1}\left|w(u)\left(e^{i\xi_k s_ku}-e^{iLu}\right)\right|\mathbb{E}\left[|X_1|^{p'} \right] \\&\leq  C_{p'} |s_k\xi_k-L| \|w\|_\infty \mathbb{E}\left[ |X_1|^{p'} \right] \, ,
\end{align*}
which converges to zero by the continuity of $w$ on $[0,1]$ and the assumption that $s_k\xi_k$ converges to $L.$

Similarly, in the case where $(X (t))_{t\in\mathbb{R}}$ is a fractional Brownian motion with Hurst parameter $H,$ we again  need to show
\begin{equation*}
 \lim_{k\rightarrow\infty}\left(\mathbb{E} \left[ \left|\int_0^1w(u)\left(e^{i\xi_k s_ku}-e^{iLu}\right) \, dX (u) \right|^p \right] \right)^{\nicefrac{1}{p}}=0 \, .
\end{equation*}
However, fractional Brownian motion is not a semi-martingale so we cannot apply Burkholder's theorem as we did in the proof of Theorem \ref{thm: alpha stable}. Instead, we use the Young-L\'oeve estimate \cite{young1936} which states that if $x(u)$ is any (deterministic) function with bounded variation, and $y(u)$ is any function which is $\alpha$-H\"older continuous, $0<\alpha<1,$ then 
\begin{equation*}
\int_0^1 x(u) \, dy(u) 
\end{equation*} 
is well-defined as the limit of Riemann sums and 
\begin{equation*}
\left|\int_0^1x(u) \, dy(u)-x(0)\left(y(1)-y(0)\right)\right|\leq C_{\alpha}\|x\|_{BV}\|y\|_\alpha \, ,
\end{equation*}
where $\|\cdot\|_{BV}$ and $\|\cdot\|_\alpha$ are the bounded variation and $\alpha$-H\"older seminorms respectively.  For all $k$, the function $h_k(u) \coloneqq w(u)\left(e^{i\xi_k s_ku}-e^{iLu}\right)\coloneqq w(u) f_k(u)$ satisfies, $h_k(0)=0$ and 
\begin{equation*}\|h_k\|_{BV}\leq \|w\|_\infty \|f_k\|_{BV}+\|w\|_{BV}\|f_k\|_\infty \, .
\end{equation*}
One can check that the fact that $s_k\xi_k$ converges to $L$ implies that $f_k$ converges to zero in both $\Lb^\infty$ and in the bounded variation seminorm, and that therefore that $\|h_k\|_{BV}$ converges to zero. 

It is well-known that fractional Brownian motion with Hurst parameter $H$ admits a continuous modification which is $\alpha$-H\"older continuous for any $\alpha<H.$ Therefore,
\begin{equation*}
\mathbb{E} \left|\int_0^1w(u)\left(e^{i\xi_k s_k u} - e^{i L u} \right) \, dX (u) \right|^p \leq C_\alpha^p \|h_k\|_{BV}^p \mathbb{E} \left[ \|X\|_\alpha^p \right].
\end{equation*}
Lastly, one can use the  Garsia-Rodemich-Rumsey inequality \cite{Garsia1970}, 
  to show that \begin{equation*}
\mathbb{E}[ \|X\|^p_\alpha ] <\infty \, .
\end{equation*}
for all $1<p<\infty.$ For details we refer the reader to the survey article \cite{Shevchenko2015}.
Therefore, 
\begin{equation*}
    \lim_{k\rightarrow0}
\mathbb{E} \left[ \left|\int_0^1w(u)\left(e^{i\xi_k s_k u}- e^{i L u} \right) \, dX (u) \right|^p \right] = 0
\end{equation*}
as desired.

\end{proof}

\begin{remark}
The assumption that $w$ has bounded-variation was used to justify that the stochastic integral against fractional Brownian motion was well defined as the limit of Riemann sums because of its H\"{o}lder continuity and the above mentioned result of \cite{young1936}. This allowed us to avoid the technical complexities of defining such an integral using either the Malliavin calculus or the Wick product.
\end{remark}

\section{Details of Numerical Experiments}
\begin{algorithm}
Initialize $V = 0$, $t = 0$\\
\While{$t < N$}{
generate $U \sim \mathcal{U}([0,1])$ \\
$V \leftarrow V - \log U$ \\
$ t = \inf \{v : \Lambda(v) < V\}$ \\
deliver $t$
}\label{alg: simulation}
\caption{Algorithm for simulating inhomogeneous Poisson point process}
\end{algorithm}

\subsection{Definition of Filters}
For all the numerical experiments, we take the window function $w$ to be the smooth bump function 
$$w(t) =\begin{cases} \exp\left(- \frac{1}{4t - 4t^2}\right), & t \in (0, 1)\\
0, &\text{otherwise}.
\end{cases}$$
Therefore for $\gamma=(s,\xi)$, our filters are given by 
$$g_{\gamma}(t) = e^{i \xi t} w(t)=\begin{cases} e^{i \xi t} e^{ - \nicefrac{s^2}{(4ts - 4t^2)}},&t\in(0,s)\\
0,&\text{otherwise}\end{cases}.  
$$

\subsection{Frequencies}
In all of our experiments, we hold the frequency, $\xi,$ which we sample uniformly at random from $(0, 2 \pi),$ constant while allowing the scale to decrease to zero.  

\subsection{Simulation of Poisson point process} 
We use the standard method to generate a realization of a Poisson point process. For Poisson point process with intensity $\lambda$, the time interval between two neighbor jumps follows exponential distribution:
$$\Delta_j \coloneqq t_j-t_{j-1}\sim \text{Exp}(\lambda) \, .$$
Therefore, taking the inverse cumulative distribution function, we sample the time interval between two neighbor jumps through:
$$\Delta_j = - \frac{\log U_j}{\lambda}, $$
where $U_j$ are i.i.d. uniform random variables on $[0,1],$
and assign the charge $A_j$ to the jump at location $t_j$.

For inhomogeneous Poisson process with intensity funciton $\lambda(t)$, we simulate the time interval based on a well-known algorithm. We, first define the cumulated intensity:
$$\Lambda(t) = \int_0^t \lambda(s) ds \, ,$$
then generate the location of jumps $t_j$ by the Algorithm \ref{alg: simulation}.

\end{document}